\documentclass[a4paper,12pt]{amsart}
\usepackage{amssymb, enumerate, mdwlist, amsthm, amsmath, bm, graphicx}
\usepackage{xcolor,hyperref}
\hypersetup{
    colorlinks=true,
    citecolor=blue,
    linkcolor=red,
    urlcolor=orange,
}
\newtheorem{thm}{Theorem}
\newtheorem{lemma}{Lemma}
\def\erf{\mathop{\mathrm{erf}}}
\def\N{\mathbb N}
\def\AA{\mathcal A}
\def\BB{\mathcal B}
\def\CC{\mathcal C}
\def\FF{\mathcal F}
\def\GG{\mathcal G}
\def\ec{\epsilon_c}
\def\ect{\frac{\epsilon_c}2}
\def\ecf{\frac{\epsilon_c}4}

\title{Non-trivial $3$-wise intersecting uniform families}
\author{Norihide Tokushige}
\address{University of the Ryukyus\\College of Education\\
1 Senbaru Nishihara\\Okinawa\\903-0213 (JAPAN)\\
ORCID: 0000-0002-9487-7545}
\email{hide@edu.u-ryukyu.ac.jp}
\date{\today}

\begin{document}
\maketitle
\begin{abstract}
A family of $k$-element subsets of an $n$-element set is called
3-wise intersecting if any three members in the family have non-empty 
intersection. We determine the maximum size of such families
exactly or asymptotically.
One of our results shows that for every $\epsilon>0$ there exists
$n_0$ such that if $n>n_0$ and $\frac25+\epsilon<\frac kn<\frac 12-\epsilon$
then the maximum size is $4\binom{n-4}{k-3}+\binom{n-4}{k-4}$.

\end{abstract}
\section{Introduction}
In this paper we study the maximum size of non-trivial $3$-wise
intersecting $k$-uniform families on $n$ vertices. We start with some
definitions. Let $n,k,r$ and $t$ be positive integers with $n>k>t$ and 
$r\geq 2$. Let $[n]=\{1,2,\ldots,n\}$,
and let $2^{[n]}$ and $\binom{[n]}k$ denote the power set of $[n]$ and
the set of all $k$-element subsets of $[n]$, respectively.
A family of subset $\GG\subset 2^{[n]}$ is called $r$-wise $t$-intersecting
if $|G_1\cap G_2\cap\cdots\cap G_r|\geq t$ for all 
$G_1,G_2,\ldots,G_r\in\GG$.
This family is called non-trivial $r$-wise $t$-intersecting, if
$\GG$ moreover satisfies $|\bigcap_{G\in\GG}G|<t$.
For simplicity we often omit $t$ if $t=1$, e.g., 
an $r$-wise intersecting family means an $r$-wise $1$-intersecting family.
Let $M_r^t(n,k)$ denote the maximum size of a family 
$\FF\subset\binom{[n]}k$ which is non-trivial $r$-wise $t$-intersecting.
Again for simplicity we write $M_r(n,k)$ to mean $M_r^1(n,k)$.
Hilton and Milner \cite{HM} determined $M_2(n,k)$, and 
Ahlswede and Khachatrian \cite{AK1996} determined $M_2^t(n,k)$ for all $t\geq 1$.
Recently O'Neill and Verstr\"aete \cite{OV}, and Balogh and Linz \cite{BL}
studied $M_r(n,k)$ mainly for the cases $n\gg k$.
Some other related results can be found in \cite{FKN, K, LL, M, XLZ}.
Our goal is to determine $M_3(n,k)$ exactly or asymptotically for most of
$n,k$ with $0<\frac kn<1$.

This problem is closely related to determine the maximum $p$-measure of 
non-trivial $r$-wise $t$-intersecting families, see \cite{T2005RMJ,T2010DM}. 
For a real number $p$
with $0<p<1$ and a family $\GG\subset 2^{[n]}$ we define the $p$-measure
$\mu_p(\GG)$ of $\GG$ by
\[
 \mu_p(\GG)=\sum_{G\in\GG} p^{|G|}(1-p)^{n-|G|}.
\]
Let $W_r^t(n,p)$ (and $W_r(n,p)$ for $t=1$) denote the maximum $p$-measure 
$\mu_p(\GG)$ of a family $\GG\subset 2^{[n]}$ which is non-trivial $r$-wise 
$t$-intersecting. Brace and Daykin \cite{BD} determined $W_r(n,\frac12)$.

The well-known Erd\H{o}s--Ko--Rado theorem \cite{EKR} and its $p$-measure
version state the following: if $\frac kn\leq\frac12$ and 
$\FF\subset\binom{[n]}k$ is a 2-wise intersecting family 
then $|\FF|/\binom nk\leq\frac kn$, and
if $p\leq\frac12$ and $\GG\subset 2^{[n]}$ is a 2-wise intersecting 
family then $\mu_p(\GG)\leq p$.
There are other such examples e.g., in \cite{T2005RMJ}, which suggest
the following working hypothesis:
if $\tfrac kn\approx p$ then $M_r^t(n,k)/\tbinom nk\approx W_r^t(n,p)$.
This is indeed the case if $r=3$ and $t=1$ as we will describe shortly,
and our result is corresponding to the following result on $W_3(n,p)$.
\begin{thm}[\cite{T2022}]\label{thm1}
For $0<p<1$ and $q=1-p$ we have
\begin{align*}
\lim_{n\to\infty} W_3(n,p)=\begin{cases}
       p^2 & \text{if } p\leq\frac13,\\
       4p^3q+p^4 & \text{if } \frac13\leq p\leq\frac12,\\
       p & \text{if } \frac12<p\leq\frac23,\\
       1 & \text{if } \frac23<p<1.
      \end{cases}
\end{align*}
\end{thm}

To state our result we need to provide some large non-trivial 3-wise 
intersecting families that are potential candidates for $M_3(n,k)$.
If $\frac 23<\frac kn<1$ then $\binom{[n]}k$ is non-trivial 3-wise 
intersecting, and $M_3(n,k)=\binom nk$.
From now on we assume that $\frac kn\leq\frac23$.
For integers $a<b$, let $[a,b]=\{a,a+1,\ldots,b\}$. Define non-trivial
3-wise intersecting families $\AA,\BB$ and $\CC$ by
\begin{align*}
\AA(n,k) &= \left\{F\in\tbinom{[n]}k:[2]\subset F,\,F\cap[3,k+1]\neq\emptyset\right\}
\cup\big\{[k+1]\setminus\{1\},\,[k+1]\setminus\{2\}\big\},\\
\BB(n,k) &= \left\{F\in\tbinom{[n]}k:|F\cap[4]|\geq 3\right\},\\
\CC(n,2l) &= \left\{F\in\tbinom{[n]}{2l}:1\in F,\,|F\cap[2,2l]|\geq l\right\}
\cup \big\{[2,2l]\cup\{i\}:i\in[2l+1,n]\big\},\\
\CC(n,2l+1) &= \left\{F\in\tbinom{[n]}{2l+1}:1\in F,\,|F\cap[2,2l+2]|\geq l+1\right\}
\cup \big\{[2,2l+2]\big\}.
\end{align*}
Then it follows that
\begin{align}
 |\AA(n,k)| &= \binom{n-2}{k-2}-\binom{n-k-1}{k-2}+2,\label{A}\\
 |\BB(n,k)| &= 4\binom{n-4}{k-3}+\binom{n-4}{k-4},\label{B}\\
 |\CC(n,2l)| &=\sum_{j=l}^{2l-1}\binom{2l-1}{j}\binom{n-2l}{2l-j-1}+(n-2l),\label{C2l}\\
 |\CC(n,2l+1)| &=\sum_{j=l+1}^{2l+1}\binom{2l+1}{j}\binom{n-2l-2}{2l-j}+1.\label{C2l+1}
\end{align}
If $p=\frac kn$ is fixed and $n\to\infty$ then 
$|\AA(n,k)|/\tbinom nk\to p^2$, and $|\BB(n,k)|/\tbinom nk\to 4p^3q+p^4$.
If, moreover, $\frac12<p\leq\frac23$ then $|\CC(n,k)|/\binom nk\to p$ 
(see the proof of Theorem~\ref{thmC}).
Thus, by the working hypothesis with Theorem~\ref{thm1}, we may expect that 
if $p=\frac kn$ is fixed and $n$ is sufficiently large, then
\begin{align*}
M_3(n,k)\approx\begin{cases}
       |\AA(n,k)| & \text{if } p\lessapprox\frac13,\\
       |\BB(n,k)| & \text{if } \frac13\lessapprox p\lessapprox\frac12,\\
       |\CC(n,k)| & \text{if } \frac12\lessapprox p\lessapprox\frac23,\\
       \binom nk & \text{if } \frac23 < p < 1.
      \end{cases}
\end{align*}
We will see that the aforementioned formula is roughly correct, 
but the actual statement we prove is a little more complicated.
If $k=3$ (and $n\geq 4$) then we have
$\AA(n,3)=\BB(n,3)=\CC(n,3)=\binom{[4]}3$ and it is easy to see that
$M_3(n,3)=4$. From now on we assume that $k\geq 4$.

\begin{thm}\label{thmA}
 We have
\begin{align*}
M_3(n,k)=\begin{cases}
       |\AA(n,k)| & \text{if } k\geq 9\text{ and } n\geq 3k-2,\\
       |\BB(n,k)| & \text{if } k\geq 4\text{ and } \tfrac52(k-1)\leq n\leq 3(k-1).
      \end{cases}
\end{align*}
\end{thm}
We record some results on $M_3(n,k)$ for $4\leq k\leq 8$, which are not
necessarily included in Theorem~\ref{thmA}.
\begin{align*}
M_3(n,4)&=|\BB(n,4)|  \quad\text{if } n\geq 7,\\
M_3(n,5)&=|\BB(n,5)|  \quad\text{if } n\geq 9,\\
M_3(n,6)&=
\begin{cases}
       |\AA(n,k)| & \text{if } n\geq 21,\\
       |\BB(n,k)| & \text{if } 11\leq n\leq 20,\\
\end{cases}\\
M_3(n,7)&=
\begin{cases}
       |\AA(n,k)| & \text{if } n\geq 21,\\
       |\BB(n,k)| & \text{if } 12\leq n\leq 20,\\
\end{cases}\\
M_3(n,8)&=
\begin{cases}
       |\AA(n,k)| & \text{if } n\geq 23,\\
       |\BB(n,k)| & \text{if } 15\leq n\leq 22.
\end{cases}
\end{align*}

\begin{thm}\label{thmB}
For every $\epsilon>0$ there exists $n_0\in\N$ such that for all integers 
$n$ and $k$ with $n>n_0$ and $\frac25+\epsilon<\frac kn<\frac12-\epsilon$, 
we have 
\[
 M_3(n,k)=|\BB(n,k)|.
\]
\end{thm}

\begin{thm}\label{thmC}
For every $\epsilon>0$ and every $\delta>0$,
there exists $n_0\in\N$ such that for all integers $n$ and $k$ with
$n>n_0$ and $\frac12+\epsilon<\frac kn<\frac23-\epsilon$, we have
\[
 (1-\delta)\binom{n-1}{k-1}\leq M_3(n,k)<\binom{n-1}{k-1}.
\]
\end{thm}
Indeed we prove $(1-\delta)\binom{n-1}{k-1}<|\CC(n,k)|$ 
under the assumptions in Theorem~\ref{thmC}, from which it follows that
\[
|\CC(n,k)|\leq M_3(n,k)<\frac1{1-\delta}\,|\CC(n,k)|. 
\]

In section 2 we deduce Theorem~\ref{thmA} from the Ahlswede--Khacatrian 
theorem (Theorem~\ref{thm:AK}) concerning non-trivial 2-wise 2-intersecting 
families. 
Then in section 3 we show a stronger stability version (Theorem~\ref{thmB+}) 
of Theorem~\ref{thmB} based on the corresponding stability result 
(Theorem~\ref{thm:W}) on $W_3(n,p)$ from \cite{T2022}.
Finally in section 4 we prove Theorem~\ref{thmC} by estimating the size of 
the family $\CC(n,k)$ using a variant of the de Moivre--Laplace theorem 
(Lemma~\ref{lemma:deML}).

\section{Proof of Theorem~\ref{thmA}}

We extend the construction of $\BB(n,k)$ and define
\[
 \BB_i(n,k)=\{F\in\binom{[n]}k:|F\cap[2+2i]|\geq t+i\}.
\]
Note that $\BB_1(n,k)=\BB(n,k)$. Note also that $\BB_i(n,k)$ is
non-trivial 2-wise 2-intersecting for all $i\geq 1$, but 
non-trivial 3-wise intersecting only if $i=1$.
Ahlswede and Khachatrian determined the maximum size of non-trivial
2-wise $t$-intersecting families in $\binom{[n]}k$ completely.
Here we include their result for the case $t=2$.

\begin{thm}[Ahlswede--Khachatrian \cite{AK1996}]\label{thm:AK} 
The maximize size $M_2^2(n,k)$ of non-trivial $2$-wise 
$2$-intersecting families in $\binom{[n]}k$ is given by the following.
\begin{enumerate}
\item[\rm (i)] If $2(k-1)<n\leq 3(k-1)$ then $M_2^2(n,k)=\max_{i\geq 1}|\BB_i(n,k)|$.
If moreover $M_2^2(n,k)=|\BB_i(n,k)|$ then 
$(2+\frac1{i+1})(k-1)\leq n\leq (2+\frac 1i)(k-1)$.
\item[\rm (ii)] If $3(k-1)<n$ and $k\leq 5$ then $M_2^2(n,k)=|\BB(n,k)|$.
\item[\rm (iii)] If $3(k-1)<n$ and $k\geq 6$ then $M_2^2(n,k)=\max\{|\AA(n,k),|\BB(n,k)|\}$.
\end{enumerate} 
Moreover, if $\FF\subset\binom{[n]}k$ is a non-trivial $2$-wise 
$2$-intersecting family with $|\FF|=M_2^2(n,k)$, then 
$\bigcap_{F\in\FF}F=\emptyset$.
\end{thm}

\begin{lemma}
 If $k\geq 4$ and $2k\leq n\leq 3(k-1)$ then $|\AA(n,k)|<|\BB(n,k)|$.
\end{lemma}
\begin{proof}
By \eqref{A} and \eqref{B},
the inequality $|\AA(n,k)|<|\BB(n,k)|$ is equivalent to
\[
 \binom{n-2}{k-2}-\binom{n-k-1}{k-2}+2<4\binom{n-4}{k-3}+\binom{n-4}{k-4}.
\]
Using $\binom{n-2}{k-2}=\binom{n-4}{k-2}+2\binom{n-4}{k-3}+\binom{n-4}{k-4}$
and $\binom{n-4}{k-3}=\frac{k-2}{n-k-1}\binom{n-4}{k-2}$,
the inequality stated above is equivalent to
\[
 \binom{n-4}{k-2}\frac{n-3(k-1)}{n-k-1}+2<\binom{n-k-1}{k-2}.
\]
Since $n\leq3(k-1)$ the LHS is at most 2, while the RHS satisfies
\[
\binom{n-k-1}{k-2} \geq \binom{k-1}{k-2}=k-1\geq 3,
\]
because $n\geq 2k$ and $k\geq 4$. 
\end{proof}

\begin{lemma}\label{lemma:A>B}
 If $k\geq 9$ and $n\geq 3k-2$ then $|\AA(n,k)|>|\BB(n,k)|$.
\end{lemma}
\begin{proof}
We need to show that
\begin{align}\label{eq:A>B}
 \binom{n-4}{k-2}\frac{n-3(k-1)}{n-k-1}+2>\binom{n-k-1}{k-2}. 
\end{align}
For $k\geq 18$ we claim a stronger inequality
$\binom{n-4}{k-2}\frac{n-3(k-1)}{n-k-1}\geq\binom{n-k-1}{k-2}$, that is,
\begin{align}\label{eq:g0}
 g(n,k):=\frac{(n-4)(n-5)\cdots(n-k)(n-3k+3)}{(n-k-1)(n-k-2)\cdots(n-2k+2)}\geq 1. 
\end{align}
First we consider the case $n=3k-2$, and let 
\[
 f(k):=g(3k-2,k)=\frac{(3k-6)\cdots(2k-2)}{(2k-3)\cdots k}.
\]
Then we have
\[
 \frac{f(k)}{f(k+1)}=\frac{(2k-1)^2(2k-2)^2}{(3k-3)(3k-4)(3k-5)k}<1
\]
for $k\geq 4$.
Since $f(9)>1$ we have $1<f(9)<f(10)<\cdots$, and so
\begin{align}\label{eq:g1}
g(3k-2,k)>1   
\end{align}
for $k\geq 9$.

Next fix $k\geq 18$ and let
\begin{align}
h(n)&:=g(n+1,k)-g(n,k)\nonumber\\
&
=
\frac{(n-4)\cdots(n-k+1)}{(n-k-1)\cdots(n-2k+3)}
\left(\frac{(n-3)(n-3k+4)}{n-k}-\frac{(n-k)(n-3k+3)}{n-2k+2}\right)
\nonumber
\\
&=c_n(k-2)\left((5-k)n+3k^2-15k+12\right), \label{eq:g2}
\end{align}
where 
\[
 c_n:=\frac{(n-4)\cdots(n-k+1)}{(n-k-1)\cdots(n-2k+3)}
\cdot\frac1{(n-k)(n-2k+2)}
>0.
\]
Then $h(n)$ is decreasing in $n$, and $h(3k+1)=c_{3k+1}(k-2)(17-k)<0$,
so we have $h(n)<0$ for $n\geq 3k+1$. 
Thus it follows that
\begin{align*}
 g(3k+1,k)>g(3k+2,k)>\cdots 
\end{align*}
for $k\geq 18$. 
We also have $g(n,k)\to 1$ for $k$ fixed and $n\to\infty$, because
both the denominator and numerator of \eqref{eq:g0} consist of
$k-2$ linear terms in $n$. Thus it follows that
$g(n,k)\geq 1$ for $n\geq 3k+1$.
For $n=3k-2,3k-1,3k$, we can verify 
$h(n)>0$ by \eqref{eq:g2}. This together with \eqref{eq:g1} yields
\begin{align*}
1< g(3k-2,k)< g(3k-1,k)< g(3k,k)< g(3k+1,k). 
\end{align*}
Therefore we have $g(n,k)\geq 1$ for all $n\geq 3k-2$ and $k\geq 18$.

For the remaining cases $9\leq k\leq 17$ we can verify \eqref{eq:A>B}
directly.
\end{proof}

Now we prove Theorem~\ref{thmA}.
Observe that if $\FF\subset\binom{[n]}k$ is non-trivial 
3-wise intersecting, then it is necessarily non-trivial 2-wise 
2-intersecting. This yields that $M_3(n,k)\leq M_2^2(n,k)$.

First suppose that $k\geq 9$ and $n\geq 3k-2$.
In this case, by (iii) of Theorem~\ref{thm:AK} with Lemma~\ref{lemma:A>B},
we have
\[
 M_3(n,k)\leq M_2^2(n,k)=\max\{|\AA(n,k)|,|\BB(n,k)|\}=|\AA(n,k)|.
\]
Since $\AA(n,k)$ is non-trivial 3-wise intersecting, we have
$M_3(n,k)=|\AA(n,k)|$.

Next suppose that $k\geq 4$ and $\frac 52(k-1)\leq n\leq 3k-3$.
In this case, by (i) of Theorem~\ref{thm:AK}, we have
\[
  M_3(n,k)\leq M_2^2(n,k)=|\BB_1(n,k)|=|\BB(n,k)|.
\]
Since $\BB(n,k)$ is non-trivial 3-wise intersecting, we have
$M_3(n,k)=|\BB(n,k)|$. This completes the proof of Theorem~\ref{thmA}. \qed

\bigskip
The results on $M_3(n,k)$ for $4\leq k\leq 8$ follows from 
Theorem~\ref{thm:AK} and comparing $|\AA(n,k)|$ and $|\BB(n,k)|$
using \eqref{eq:A>B}.

\section{Proof of Theorem~\ref{thmB}}
A family $\FF\subset 2^{[n]}$ is called shifted if
$F\in\FF$ and $F\cap\{i,j\}=\{j\}$, then $(F\setminus\{j\})\cup\{i\}\in\FF$
for all $1\leq i<j\leq n$.
Theorem~\ref{thmB} is an immediate consequence from the following stronger
stability result.

\begin{thm}\label{thmB+}
For every $\epsilon>0$ there exist $\gamma>0$ 
and $n_0\in\N$ such that for all positive integers $n$ and $k$ with 
$n>n_0$ and $\frac25+\epsilon<\frac kn<\frac12-\epsilon$ 
the following holds: if $\FF\subset\binom{[n]}k$
is a shifted non-trivial $3$-wise intersecting family with
$\FF\not\subset\BB(n,k)$, then $|\FF|<(1-\gamma)|\BB(n,k)|$.
\end{thm}

We deduce Theorem~\ref{thmB+} from the next technical lemma.

\begin{lemma}\label{laB+}
For every $\epsilon>0$ and every $p$
with $\frac25+\epsilon<p<\frac12-\epsilon$ there exist $\gamma>0$ 
and $n_0\in\N$ such that for all positive integers $n$ and $k$ with 
$n>n_0$ and $|\frac kn-p|<\frac{\epsilon}2$ the following holds: 
if $\FF\subset\binom{[n]}k$
is a shifted non-trivial $3$-wise intersecting family with
$\FF\not\subset\BB(n,k)$, then $|\FF|<(1-\gamma)|\BB(n,k)|$.
\end{lemma}

For real numbers $\alpha>\beta>0$, we write $\alpha\pm\beta$ to mean the 
open interval $(\alpha-\beta,\alpha+\beta)$. 

\begin{proof}[Proof of Theorem~\ref{thmB+}]
Here we prove Theorem~\ref{thmB+} assuming Lemma~\ref{laB+}.
Let $\epsilon>0$ be given.
Let $I=(\frac25+\epsilon,\frac12-\epsilon)$. 
Noting that $\frac12-\frac25=\frac1{10}$ we can divide $I$ into
at most $\frac1{10\epsilon}$ small intervals $I_p:=p\pm\frac{\epsilon}2$.
More precisely, by choosing real numbers 
$\frac25+\frac{3\epsilon}2\leq p_1<p_2<\cdots<p_N\leq\frac12-\frac{3\epsilon}2$ appropriately, where 
$N\leq\frac1{10\epsilon}$, we can cover $I$ by 
$\bigcup_{1\leq i\leq N}I_{p_i}$.
For every $p\in I$ there exists $i$ such that $p\in I_i$.
Apply Lemma~\ref{laB+} with this $p$, then it provides $\gamma=\gamma_i$ 
and $n_0=n_0(i)$. Finally let $\gamma=\min_{1\leq i\leq N}\gamma_i$ and 
$n_0=\max_{1\leq i\leq N} n_0(i)$. Then Theorem~\ref{thmB+} follows from
Lemma~\ref{laB+}.
\end{proof}
Now we prove Lemma~\ref{laB+}. For the proof we need the $p$-measure 
counterpart of Lemma~\ref{laB+}.
To state the result we define a non-trivial 3-wise intersecting family 
$\BB(n)\subset 2^{[n]}$ by
\[
 \BB(n)=\{F\in 2^{[n]}:|F\cap[4]|\geq 3\}.
\]
Then $\mu_p(\BB(n))=4p^3q+p^4=:f(p)$,
where $q=1-p$. 
\begin{thm}[\cite{T2022}]\label{thm:W}
 Let $\frac25\leq p\leq\frac12$, and let $\GG\subset 2^{[n]}$ be a shifted
non-trivial $3$-wise intersecting family. If $\GG\not\subset\BB(n)$ then
\[
 \mu_p(\GG)<\mu_p(\BB(n))-0.0018\leq(1-0.00576)f(p).
\]
\end{thm}

We deduce Lemma~\ref{laB+} from Theorem~\ref{thm:W}.
To this end we assume the negation of Lemma~\ref{laB+}, 
and derive a contradiction to Theorem~\ref{thm:W}.
The main idea is that if we have a counterexample family 
$\FF\subset\binom{[n]}k$ to Lemma~\ref{laB+}, then the family 
$\GG\subset 2^{[n]}$ consisting of all superset of $F\in\FF$ satisfies the 
assumptions of Theorem~\ref{thm:W},  but its $p$-measure is too large.

Suppose that Lemma~\ref{laB+} fails at some $\epsilon_c>0$ and some 
$p_c$ with $\frac25+\ec<p_c<\frac12-\ec$, 
where `$c$' stands for counterexample.
Fix these $\epsilon_c$ and $p_c$. Set
\begin{align}\label{WM1}
\gamma_c=\frac{0.00576}4,\quad I_c=p_c\pm \ect.
\end{align}
Since $f(p)$ is continuous in $p$ we can find $\epsilon_1$ with
$0<\epsilon_1\leq\ecf$ such that 
\begin{align}\label{WM2}
 (1-3\gamma_c)f(p)>(1-4\gamma_c)f(p+\delta)
\end{align}
for all $0<\delta\leq\epsilon_1$ and all $p\in I_c$.
Define an open interval
\[
J_{n,p}=\left(\left(p-\epsilon_1\right)n,\left(p+\epsilon_1\right)n\right)\cap\N. 
\]
By the concentration of the binomial distribution 
we can choose $n_1$ so that 
\begin{align}\label{WM3}
\sum_{j\in J_{n,p}}\binom njp^jq^{n-j}>\frac{1-3\gamma_c}{1-2\gamma_c}
\end{align}
for all $n>n_1$ and all $p\in I_0:=p_c\pm\frac34\ec$.
As $|\BB(n,k)|/\binom nk\to f(p)$ if $p=\frac kn$ is fixed and $n\to\infty$,
we can find $n_2$ such that
\begin{align}\label{WM4}
(1-\gamma_c)\frac{|\BB(n,k)|}{\binom nk}>(1-2\gamma_c)f\left(\frac kn\right).
\end{align}
for all $n>n_2$ and all $k$ with $\frac kn\in I_c$.
Let $n_0=\max\{n_1,n_2\}$.

With these $\ec,p_c,\gamma_c,n_0$ we can choose $n_c$, $k_c$ and $\FF_c$
with $n_c>n_0$ and $\frac {k_c}{n_c}\in I_c$ such that 
$\FF_c\subset\binom{[n_c]}{k_c}$ is a counterexample to Lemma~\ref{laB+}, 
that is, $\FF_c$ is shifted, non-trivial 3-wise intersecting, but 
\begin{align}\label{WM4.5}
|\FF_c|\geq(1-\gamma_c)|\BB(n,k)|. 
\end{align}
Fix these $n_c,k_c$ and $\FF_c$. By \eqref{WM4.5} and \eqref{WM4} we have
\begin{align}\label{WM5}
  |\FF_c|>(1-2\gamma_c)f\left(\frac {k_c}{n_c}\right)\binom {n_c}{k_c}.
\end{align}
Let $p=\frac{k_c}{n_c}+\epsilon_1$. Since
\[
p_c-\frac34\ec<\left(p_c-\ect\right)+\epsilon_1< p<\left(p_c+\ect\right)+\epsilon_1\leq p_c+\frac34\ec,
\]
we have $p\in I_0$. Also it follows from $k_c=(p-\epsilon_1)n_c$ that
$J_{n_c,p}\subset[k_c,n_c]$. Define a shifted non-trivial 3-wise 
intersecting family $\GG\subset 2^{[n]}$ by
\[
 \GG=\bigcup_{j=k_c}^{n_c}\nabla_j(\FF_c),
\]
where $\nabla_j$ denotes the $j$-th upper shadow, that is,
$\nabla_j(\FF_c)=\{H\in\binom{[n]}j:H\supset\exists F\in\FF_c\}$.
It follows from \eqref{WM5} and the Kruskal--Katona theorem that
\begin{align}\label{WM6}
|\nabla_j(\FF_c)|\geq
(1-2\gamma_c)f\left(\frac {k_c}{n_c}\right)\binom {n_c}{j}
\end{align}
for $j\in J_{n_c,p}$, see Claim~6 in \cite{T2010DM} for a detailed proof.
Then we have
\begin{align*}
 \mu_{p}(\GG)&\geq \sum_{j\in J_{n_c,p}}|\nabla_j(\FF)|p^jq^{n-j}\\
&\geq (1-2\gamma_c)f\left(\frac{k_c}{n_c}\right)\sum_{j\in J_{n_c,p}}
\binom {n_c}jp^jq^{n-j}
\qquad\text{by \eqref{WM6}}\\
&>(1-2\gamma_c)f\left(\frac{k_c}{n_c}\right)\frac{1-3\gamma_c}{1-2\gamma_c}
\qquad\text{by \eqref{WM3}}\\
&=(1-3\gamma_c)f\left(\frac{k_c}{n_c}\right)\\
&>(1-4\gamma_c)f\left(\frac{k_c}{n_c}+\epsilon_1\right)
\qquad\text{by \eqref{WM2}}\\
&=(1-4\gamma_c)f(p)\\
&=(1-0.00576)f(p),
\qquad\text{by \eqref{WM1}}
\end{align*}
which contradicts Theorem~\ref{thm:W}.
This completes the proof of Lemma~\ref{laB+}, and so Theorem~\ref{thmB+}
(and Theorem~\ref{thmB}). \qed

\section{Proof of Theorem~\ref{thmC}}
For the proof we prepare some technical estimations.
Let $\frac12<p<\frac23$ and $q=1-p$. Let $k=pn$. Fix a constant $c>0$, 
and let
\[
 J_{n,p}=\{j\in\N:|j-p^2n|\leq c\sqrt{n}\}.
\]
For $j\in J_{n,p}$ we estimate
\[
 \theta_j(n,p)=\frac{\binom {pn}j\binom{n-pn}{pn-j}}{\binom n{pn}}
=\frac{\binom kj\binom{n-k}{k-j}}{\binom nk}.
\]
\begin{lemma}\label{lemma:deML}
We have
\[
 \theta_j(n,p)=\frac1{pq\sqrt{2\pi n}}\exp
\left(-\frac1{2p^4q^4n^2}(j-p^2n)^2
\left(p^2q^2n-(1-2p)^2(j-p^2n)\right)+r_{n,p}(j)\right),
\]
where $\max_{j\in J_{n,p}}|r_{n,p}(j)|\to 0$ as $n\to\infty$.
\end{lemma}
\begin{proof}
This is a variant of the de Moivre--Laplace theorem, and it
follows from a routine but tedious calculus.
Here we include the outline. By Stirling's formula we have
\[
\binom ab\sim\sqrt{\frac a{2\pi b(a-b)}}\,\frac{a^a}{b^b(a-b)^{a-b}}, 
\]
and so
\begin{align*}
 \theta_j(n,p)&\sim
\sqrt{\frac n{2\pi pq k(n-k)}}\,
(pk)^{k+1}(q(n-k))^{n-k+1}j^{-j-\frac12}(qn-k+j)^{-qn+k-j-\frac12}\\
&\qquad \times(pn-j)^{-pn+j-\frac12}(k-j)^{-k+j-\frac12}.
\end{align*}
Noting that $k+1=(j+\frac12)+(k-j+\frac12)$ and
$n-k+1=(qn-k+j+\frac12)+(pn-j+\frac12)$ we can rewrite $\theta_j(n,p)$ as
\begin{align}\label{eq:theta=c exp(-A)}
 \theta_j(n,p)\sim\frac1{pq\sqrt{2\pi n}}\exp(-A), 
\end{align}
where
\begin{align*}
 A&=(j+\tfrac12)\log\tfrac j{pk}+(k-j+\tfrac12)\log\tfrac{k-j}{pk}
+(qn-k+j+\tfrac12)\log\tfrac{qn-k+j}{q(n-k)}\\
&\qquad +(qn-j+\tfrac12)\log\tfrac{pn-j}{q(n-k)}.
\end{align*}
Now, to use $\log(1+x)=x-\frac{x^2}2+O(x^3)$, we recall $j\sim pk$ and we
write
\begin{align*}
 \log\tfrac j{pk}&=\log(1+\tfrac{j-pk}{pk}),\,&
 \log\tfrac {k-j}{pk}&=\log\tfrac qp+\log(1-\tfrac{j-pk}{qk}),\\
 \log\tfrac {qn-k+j}{q(n-k)}&=\log(1+\tfrac{j-pk}{q(n-k)}),\,&
 \log\tfrac {pn-j}{q(n-k)}&=\log\tfrac pq+\log(1-\tfrac{j-pk}{pk}).
\end{align*}
Thus we have e.g., 
$\log\frac j{pk}\sim(\frac{j-pk}{pk})-\frac12(\frac{j-pk}{pk})^2$.
Substituting these approximations into $A$ and rearranging, we obtain
the desired estimation.
\end{proof}
Let $\erf(z)$ denote the error function, that is,
\[
 \erf(z)=\frac 2{\sqrt\pi}\int_0^z\exp(-x^2)\,dx.
\]
\begin{lemma}\label{lemma sum theta}
 We have
\[
 \lim_{n\to\infty}\sum_{j\in J_{n,p}}\theta_j(n,p)=\erf\left(\frac{3c}{\sqrt{2}p}\right).
\]
\end{lemma}
\begin{proof}
Let $z=j-p^2 n$. Then, by Lemma~\ref{lemma:deML}, we can write 
$\theta_j(n,p)$ as in \eqref{eq:theta=c exp(-A)}, where
\[
 A=\frac1{2p^4q^4n^2}z^2(p^2q^2n-(1-2p)^2z).
\]
Thus it follows that
\begin{align*}
\sum_{j\in J_{n,p}} \theta_j(n,p)&\sim \int_{-c\sqrt n}^{c\sqrt n}
\frac1{pq\sqrt{2\pi n}}
\exp\left(-\frac1{2p^4q^4n^2}z^2(p^2q^2n-(1-2p)^2z)\right)dz.
\end{align*}
Finally, by changing $z=\sqrt n x$, we get
\[
\lim_{n\to\infty}\sum_{j\in J_{n,p}}\theta_j(n,p) =\lim_{n\to\infty} \int_{-c}^{c}
\frac1{pq\sqrt{2\pi}}
\exp\left(-\frac{x^2}{2p^2q^2}+\frac{(1-2p)^2 x^3}{2p^4q^4\sqrt n}\right)dx
=\erf\left(\frac{3c}{\sqrt2p}\right). \qedhere
\]
\end{proof}

Now we prove Theorem~\ref{thmC}.
Frankl \cite{Fshift} proved that 
if $\frac kn\leq \frac 23$ then the maximum size of 
(not necessarily non-trivial) 3-wise intersecting families 
$\FF\subset\binom{[n]}k$ is $\binom{n-1}{k-1}$, and moreover this bound 
is attained only if the family $\FF$ is trivial, that is, 
$\bigcap_{F\in\FF}F\neq\emptyset$. 
This means that $M_3(n,k)<\binom{n-1}{k-1}$. Thus the remaining part of the 
proof below is to give a lower bound for $M_3(n,k)$.

Let $\epsilon>0$ and $\delta>0$ be given.
We need some constants to define $n_0$. First choose $c$ so that
\[
\min_{p\in[\frac12,\frac23]}\erf\left(\frac{3c}{\sqrt 2 p}\right)
=\erf\left(\frac{9c}{2\sqrt 2}\right)>1-\frac{\delta}2.
\]
Then, by Lemma~\ref{lemma sum theta}, 
we can choose $n_1$ so that if $n>n_1$ then 
\begin{align}\label{eq:>1-delta}
 \sum_{j\in J_{n,p}}\theta_j(n,p)>\sqrt{1-\delta}
\end{align}
for all $p$ with $\frac12\leq p\leq\frac23$.
Next choose $n_2$ so that if $\frac12+\epsilon<p<\frac23-\epsilon$,
$n>n_2$, and $k=pn$, 
then $\frac k2<pk-c\sqrt n$ and $pk+c\sqrt n<k-1$, or equivalently,
\[
 \left(p-\frac12\right)p\sqrt n>c \text{ and }
 c\sqrt n <pq n-1.
\]
We need this condition to guarantee $J_{n,p}\subset[k/2,k-1]$.
Finally let $\delta_1,\delta_2>0$ be sufficiently small constants
such that $(1-\delta_1)^2/(1+\delta_2)>\sqrt{1-\delta}$, and choose
$n_3$ so that if $n>n_3$ then
\[
 p-\frac jn\geq p-p^2-\frac c{\sqrt n}>pq(1-\delta_1),
\]
and
\[
1-2p+\frac jn+\frac 1n\leq 1-2p+p^2+\frac c{\sqrt n}+\frac1n<q^2(1+\delta_2).
\]
Then let $n_0:=\max\{n_1,n_2,n_3\}$. 

Now let $n$ and $k$ be given, and let $p=\frac kn$. By the assumptions of
the theorem we have $n>n_0$ and $\frac12+\epsilon<p<\frac 23-\epsilon$.

First suppose that $k$ is even. In this case, by \eqref{C2l}, we have
\[
 |\CC(n,k)|>\sum_{j=k/2}^{k-1}\binom{k-1}j\binom{n-k}{k-j-1}.
\]
The summands in the RHS is
\begin{align*}
\binom{k-1}j\binom{n-k}{k-j-1}
&= \frac{k-j}k\binom kj \cdot
\frac{k-j}{n-2k+j+1}\binom{n-k}{k-j}\\
&=\frac {p-\frac jn}p\binom kj \cdot \frac{p-\frac jn}{1-2p+\frac jn+\frac1n}
\binom{n-k}{k-j}.
\end{align*} 
For $j\in J_{n,p}$ we have
\[
\frac{p-\frac jn}p \cdot \frac{p-\frac jn}{1-2p+\frac jn+\frac1n}
>\frac{\left(pq(1-\delta_1)\right)^2}{pq^2(1+\delta_2)}
=p\frac{(1-\delta_1)^2}{1+\delta_2}>p\sqrt{1-\delta}.
\]
Thus 
\[
\binom{k-1}j\binom{n-k}{k-j-1}> 
p\sqrt{1-\delta}\binom kj\binom{n-k}{k-j}.
\]
Consequently we have
\begin{align*}
 |\CC(n,k)|&>\sum_{j\in J_{n,p}}\binom{k-1}{j}\binom{n-k}{k-j-1} \\
&>p\sqrt{1-\delta}\sum_{j\in J_{n,p}}\binom kj\binom{n-k}{k-j} \\
&=p\sqrt{1-\delta}\binom nk\sum_{j\in J_{n,p}}\theta_j(n,p)\\
&>(1-\delta)\binom{n-1}{k-1}, 
\end{align*}
where we use \eqref{eq:>1-delta} for the last inequality.

Next suppose that $k$ is odd. In this case, by \eqref{C2l+1}, we have
\[
 |\CC(n,k)|>\sum_{j=(k+1)/2}^{k}\binom kj\binom{n-k-1}{k-j-1}.
\]
The summands in the RHS is 
$\binom kj\binom{n-k-1}{k-j-1}=\binom kj\frac{k-j}{n-k}\binom{n-k}{k-j}$.
For $j\in J_{n,p}$ we have
\[
 \frac{k-j}{n-k}=\frac{p-\frac jn}{1-p}>
\frac{pq(1-\delta_1)}{q}=p(1-\delta_1)>p\sqrt{1-\delta},
\]
and 
\[
\binom kj\binom{n-k-1}{k-j-1}>p\sqrt{1-\delta}\binom kj\binom{n-k}{k-j}.
\]
Thus we have $|\CC(n,k)|>(1-\delta)\binom{n-1}{k-1}$ just
as in the previous case. This completes the proof of Theorem~\ref{thmC}.
\qed

\section*{Acknowledgment}
I thank both referees for their careful reading and many helpful 
suggestions. This research was supported by JSPS KAKENHI Grant No.~18K03399.

\end{document}